\documentclass[11pt,a4paper]{article}

\usepackage{amsmath}
\usepackage{amssymb}
\usepackage{mathrsfs}
\usepackage{amsthm}

\newtheorem{theorem}{Theorem}
\newtheorem{lemma}{Lemma}

\title{Boundary Observer for Space and Time Dependent Reaction-Advection-Diffusion Equations}
\author{
        Agus Hasan\\
        Center for Unmanned Aerial Vehicles\\
        The Maersk McKinney M\o ller Institute\\
        University of Southern Denmark\\
        E: agha@mmmi.sdu.dk\\
}

\begin{document}
\maketitle

\begin{abstract}
This paper presents boundary observer design for space and time dependent reaction-advection-diffusion equations using backstepping method. The method uses only a single measurement at the boundary of the systems. The existence of the observer kernel equation is proved using the method of successive approximation.
\end{abstract}

\section{Introduction}

Many physical phenomena can be described by partial differential equations (PDEs). Some examples include model of flexible cable in an overhead crane \cite{Brigit}, automated managed pressure drilling \cite{Agus1,Agus2,Hasan1}, and battery management systems \cite{Gu,Shuxia1}. Stabilization of PDEs with boundary control is considered as a challenging topic. After the introducing of the backstepping method, it becomes an emerging area \cite{Krstic}. The backstepping method has been successfully used for estimation  and control of many PDEs, such as the Korteweg-de Vries equation \cite{Edu1,A3}, Benjamin-Bona-Mahony equation \cite{A4}, Schrodinger equation \cite{Kr1}, Ginzburg-Landau equation \cite{Aam}, and 2$\times$2 linear hyperbolic PDEs \cite{fin,deu}. In engineering, the backstepping method can be found in several applications, such as in gas coning control \cite{c8}, flow control in porous media \cite{c9}, slugging control in drilling \cite{c10}, and lost circulation control \cite{c11,c00}.

In spite of the fact that the infinite-dimensional backstepping is limited to Volterra nonlinearities \cite{Rafa1,Rafa2}, local stabilization of nonlinear PDE systems shows promising results in recent years. As an example, feedback control design for a 2$\times$2 quasilinear hyperbolic PDEs is presented in \cite{Coron}. An early research on control of Burgers' equation may be found in \cite{Krstic2}. In the paper, the authors derived the nonlinear boundary control laws that achieve global asymptotic stability using Lyapunov method. In \cite{Krstic3}, the unstable shock-like equilibrium profiles of the viscous Burgers equation was stabilized using control at the boundaries. Recent results on control of Burgers' equation could be found in \cite{Balogh,Liu,Smy}. In \cite{Aamo}, a linear coupled hyperbolic PDE-ODE system is studied. Other works in control of coupled PDE-ODE systems include control with Neumann interconnections \cite{Susto}, coupled ODE-Schrodinger equation \cite{beibei}, and adaptive control of PDE-ODE cascade systems with uncertain harmonic disturbances \cite{zaihua}. Other researchers studied coupled systems of linear PDE-ODE systems \cite{Krstic1,Daa,Hasan} and nonlinear ODE and linear PDE cascade systems \cite{ahmed}.

In this paper, we consider boundary observer design for a reaction-advection-diffusion (RAD) equation with variable coefficients in space and time. The RAD equation can be used to model chemical and biological processes in flowing medium, e.g., fluid flow through porous medium. The evolution equation combines three processes and can be derived from mass balances equation.

\section{Problem Formulation}

We consider the following space and time dependent reaction-advection-diffusion equations with mixed boundary conditions
\begin{eqnarray}
\frac{\partial u}{\partial t}(r,t) &=& D(r,t) \frac{\partial^2u}{\partial r^2}(r,t)+b(r,t)\frac{\partial u}{\partial r}(r,t)+\phi(r,t)u(r,t)\label{mol1}\\
u(0,t) &=& 0\label{mol2}\\
\frac{\partial u}{\partial r}(1,t) &=& u(1,t)+U(t)\label{mol3}
\end{eqnarray}
where $u(r,t)$ is the system state, $r\in[0,1]$ is the spatial variable, and $t\in[0,\infty)$ is the time variable. The coefficients for diffusion, advection, and reaction, are denoted by $D(r,t)\in\mathbb{C}^{2,1}\left((0,1)\times(0,\infty)\right)$, $b(r,t)\in\mathbb{C}^{1,1}\left((0,1)\times(0,\infty)\right)$, and $\phi(r,t)\in\mathbb{C}\left((0,1)\times(0,\infty)\right)$, respectively. We assume these coefficients and the input function $U(t)$ are known, and in particular $D(r,t)>0$. To simplify the equation, the advection term $\frac{\partial u}{\partial r}$ can be eliminated using a simple transformation. Let us define a new state variable
\begin{eqnarray}
c(r,t) = u(r,t)e^{\int_0^{r}\!\frac{b(\tau,t)}{2D(\tau,t)}\,\mathrm{d}\tau}\label{der}
\end{eqnarray}
The derivatives of \eqref{der} with respect to $t$ and $r$ are given by
\begin{eqnarray}
\frac{\partial u}{\partial t}(r,t) &=& \frac{\partial c}{\partial t}(r,t)e^{-\int_0^{r}\!\frac{b(\tau,t)}{2D(\tau,t)}\,\mathrm{d}\tau}\nonumber\\
&&-\left(\frac{1}{2}\int_0^{r}\!\frac{\partial}{\partial t}\left(\frac{b(\tau,t)}{D(\tau,t)}\right)\,\mathrm{d}\tau\right)c(r,t)e^{-\int_0^{r}\!\frac{b(\tau,t)}{2D(\tau,t)}\,\mathrm{d}\tau}\\
\frac{\partial u}{\partial r}(r,t) &=& \frac{\partial c}{\partial r}(r,t)e^{-\int_0^{r}\!\frac{b(\tau,t)}{2D(\tau,t)}\,\mathrm{d}\tau}-\frac{b(r,t)}{2D(r,t)}c(r,t)e^{-\int_0^{r}\!\frac{b(\tau,t)}{2D(\tau,t)}\,\mathrm{d}\tau}\\
\frac{\partial^2u}{\partial r^2}(r,t) &=& \frac{\partial^2c}{\partial r^2}(r,t)e^{-\int_0^{r}\!\frac{b(\tau,t)}{2D(\tau,t)}\,\mathrm{d}\tau}-\frac{b(r,t)}{D(r,t)}\frac{\partial c}{\partial r}(r,t)e^{-\int_0^{r}\!\frac{b(\tau,t)}{2D(\tau,t)}\,\mathrm{d}\tau}\nonumber\\
&&+\frac{b(r,t)^2}{4D(r,t)^2}c(r,t)e^{-\int_0^{r}\!\frac{b(\tau,t)}{2D(\tau,t)}\,\mathrm{d}\tau}\nonumber\\
&&-\frac{1}{2}\left(\frac{b_r(r,t)}{D(r,t)}-\frac{b(r,t)D_r(r,t)}{D(r,t)^2}\right)c(r,t)e^{-\int_0^{r}\!\frac{b(\tau,t)}{2D(\tau,t)}\,\mathrm{d}\tau}
\end{eqnarray}
Substituting these equations into \eqref{mol1}-\eqref{mol3}, yields
\begin{eqnarray}
\frac{\partial c}{\partial t}(r,t) &=& D(r,t)\frac{\partial^2c}{\partial r^2}(r,t)+\lambda(r,t)c(r,t)\label{sysmain}\\
c(0,t) &=& 0\label{sysmainbc1}\\
\frac{\partial c}{\partial r}(1,t) &=& H(t)c(1,t)+M(t)\label{sysmainbc2}
\end{eqnarray}
where
\begin{eqnarray}
\lambda(r,t) &=& \phi(r,t)-\frac{b(r,t)^2}{4D(r,t)}-\frac{b_r(r,t)}{2}+\frac{b(r,t)D_r(r,t)}{2D(r,t)}\nonumber\\
&&+\frac{1}{2}\int_0^{r}\!\frac{\partial}{\partial t}\left(\frac{b(\tau,t)}{D(\tau,t)}\right)\,\mathrm{d}\tau\\
M(t) &=& U(t)e^{\int_0^{1}\!\frac{b(\tau,t)}{2D(\tau,t)}\,\mathrm{d}\tau}\\
H(t) &=& 1+\frac{b(1,t)}{2D(1,t)}
\end{eqnarray}
The objective is to estimate the state $c(r,t)$ using only one boundary measurement $c(1,t)$.

\section{Boundary Observer Design}

We design the state observer for \eqref{sysmain}-\eqref{sysmainbc2} as the copy of the plant plus an output injection term as follow
\begin{eqnarray}
\frac{\partial \hat{c}}{\partial t}(r,t) &=& D(r,t) \frac{\partial^2\hat{c}}{\partial r^2}(r,t)+\lambda(r,t)\hat{c}(r,t)+p_1(r,t)\left(c(1,t)-\hat{c}(1,t)\right)\\
\hat{c}(0,t) &=& 0\\
\frac{\partial \hat{c}}{\partial r}(1,t) &=& H(t)\hat{c}(1,t)+M(t)+p_{10}(t)\left(c(1,t)-\hat{c}(1,t)\right)
\end{eqnarray}
where $p_1(r,t)$ and $p_{10}(t)$ are observer gains to be determined later. If we define $\tilde{c}(r,t)=c(r,t)-\hat{c}(r,t)$, then the error system is given by
\begin{eqnarray}
\frac{\partial \tilde{c}}{\partial t}(r,t) &=& D(r,t) \frac{\partial^2\tilde{c}}{\partial r^2}(r,t)+\lambda(r,t)\tilde{c}(r,t)-p_1(r,t)\tilde{c}(1,t)\label{err1}\\
\tilde{c}(0,t) &=& 0\\
\frac{\partial \tilde{c}}{\partial r}(1,t) &=& H(t)\tilde{c}(1,t)-p_{10}(t)\tilde{c}(1,t)\label{err3}
\end{eqnarray}
We employ a Volterra integral transformation
\begin{eqnarray}
\tilde{c}(r,t) = \tilde{w}(r,t) - \int_r^1\! p(r,s,t)\tilde{w}(s,t) \,\mathrm{d}s\label{trans}
\end{eqnarray}
to transform the error system \eqref{err1}-\eqref{err3} into the following target system
\begin{eqnarray}
\frac{\partial \tilde{w}}{\partial t}(r,t) &=& D(r,t)\frac{\partial^2\tilde{w}}{\partial r^2}(r,t)+\mu\tilde{w}(r,t)\label{e1}\\
\tilde{w}(0,t) &=& 0\\
\frac{\partial \tilde{w}}{\partial r}(1,t) &=& -\frac{1}{2}\tilde{w}(1,t)\label{e3}
\end{eqnarray}
where the free parameter $\mu$ can be used to set the desired rate of stability. The transformation \eqref{trans} is invertible and the transformation kernel $p(r,s,t)$ is used to find the observer gains $p_1(r,t)$ and $p_{10}(t)$.

\begin{lemma}
The error system \eqref{e1}-\eqref{e3} is exponentially stable in the $\mathbb{L}^2(0,1)$-norm under the condition
\begin{eqnarray}
\mu<-\left\{\frac{\max|D_{rr}(r,t)|}{2}+\frac{D_m^2}{\min|D(r,t)|}\right\}\label{mu}
\end{eqnarray}
where
\begin{eqnarray}
D_m=\max\left\{0,-\left(\frac{D(1,t)+D_r(1,t)}{2}\right)\right\}\label{dm}
\end{eqnarray}
\end{lemma}

\begin{proof}
Consider the following Lyapunov function
\begin{eqnarray}
W(t) &=& \frac{1}{2}\int_0^1\! \tilde{w}(r,t)^2\,\mathrm{d}r\label{lya}
\end{eqnarray}
The derivative of \eqref{lya} with respect to $t$ along \eqref{e1}-\eqref{e3} is
\begin{eqnarray}
\dot{W}(t) &=& -\left(\frac{D(1,t)+D_r(1,t)}{2}\right)\tilde{w}(1,t)^2-\int_0^1\! D(r,t)\left(\frac{\partial\tilde{w}}{\partial r}(r,t)\right)^2\,\mathrm{d}r\nonumber\\
&&+\int_0^1\! \left(\frac{D_{rr}(r,t)}{2}+\mu\right)\tilde{w}(r,t)^2\,\mathrm{d}r
\end{eqnarray}
Using Young's and Poincare's inequalities, we have
\begin{eqnarray}
\dot{W}(t) &\leq& D_m\tilde{w}(1,t)^2-\int_0^1\! D(r,t)\left(\frac{\partial\tilde{w}}{\partial r}(r,t)\right)^2\,\mathrm{d}r\nonumber\\
&&+\int_0^1\! \left(\frac{\max|D_{rr}(r,t)|}{2}+\mu\right)\tilde{w}(r,t)^2\,\mathrm{d}r\\
&\leq& \int_0^1\! \left(\frac{\max|D_{rr}(r,t)|}{2}+\frac{D_m^2}{\min|D(r,t)|}+\mu\right)\tilde{w}(r,t)^2\,\mathrm{d}r
\end{eqnarray}
where $D_m$ is given by \eqref{dm}. Choosing $\mu$ as in \eqref{mu} completes the proof.
\end{proof}

Let us calculate the derivative of \eqref{trans} with respect to $r$
\begin{eqnarray}
\frac{\partial \tilde{c}}{\partial r}(r,t) &=& \frac{\partial \tilde{w}}{\partial r}(r,t) + p(r,r,t)\tilde{w}(r,t) - \int_r^1\! p_r(r,s,t)\tilde{w}(s,t) \,\mathrm{d}s\label{x}\\
\frac{\partial^2\tilde{c}}{\partial r^2}(r,t) &=& \frac{\partial^2\tilde{w}}{\partial r^2}(r,t) + \frac{\mathrm{d}}{\mathrm{d}r}p(r,r,t)\tilde{w}(r,t) + p(r,r,t)\frac{\partial \tilde{w}}{\partial r}(r,t)\nonumber\\
&& +p_r(r,r,t)\tilde{w}(r,t)- \int_r^1\! p_{rr}(r,s,t)\tilde{w}(s,t) \,\mathrm{d}s\label{xx}
\end{eqnarray}
Furthermore, we calculate the derivative of \eqref{trans} with respect to $t$
\begin{eqnarray}
\frac{\partial \tilde{c}}{\partial t}(r,t) &=& D(r,t)\frac{\partial^2\tilde{w}}{\partial r^2}(r,t)+\mu\tilde{w}(r,t)-\int_r^1\! \mu p(r,s,t)\tilde{w}(s,t) \,\mathrm{d}s\nonumber\\
&&-p(r,1,t)D(1,t)\frac{\partial\tilde{w}}{\partial r}(1,t)+p(r,r,t)D(r,t)\frac{\partial\tilde{w}}{\partial r}(r,t)\nonumber\\
&&+\left(p(r,1,t)D(1,t)\right)_s\tilde{w}(1,t)-\left(p(r,r,t)D(r,t)\right)_s\tilde{w}(r,t)\nonumber\\
&&-\int_r^1\! \left(p(r,s,t)D(s,t)\right)_{ss}\tilde{w}(s,t) \,\mathrm{d}s\nonumber\\
&&-\int_r^1\! p_t(r,s,t)\tilde{w}(s,t) \,\mathrm{d}s\label{t}
\end{eqnarray}
Substituting \eqref{xx} and \eqref{t} into \eqref{err1}-\eqref{err3}, the following system need to be satisfied
\begin{eqnarray}
p_t(r,s,t) &=& D(r,t)p_{rr}(r,s,t)-\left(D(s,t)p(r,s,t)\right)_{ss}\nonumber\\
&&-\left(\mu-\lambda(r,t)\right) p(r,s,t)\label{ker1}\\
\frac{\mathrm{d}}{\mathrm{d}r}p(r,r,t)&=&-\frac{D_r(r,t)}{2D(r,t)}p(r,r,t)+\frac{\left(\mu-\lambda(r,t)\right)}{2D(r,t)}\label{ker2}\\
p(0,s,t) &=& 0\label{ker3}
\end{eqnarray}
Furthermore, the state observer gains are obtained as
\begin{eqnarray}
p_1(r,t)&=&-\frac{1}{2}p(r,1,t)D(1,t)-\left(p(r,1,t)D(1,t)\right)_s\label{gain1}\\
p_{10}(t) &=& \frac{1}{2}+H(t)-p(1,1,t)\label{gain2}
\end{eqnarray}

The second boundary condition \eqref{ker2} can be simplified further. First, multiplying it by $\sqrt{2D(r,t)}$, we have
\begin{eqnarray}
\sqrt{2D(r,t)}\frac{\mathrm{d}}{\mathrm{d}r}p(r,r,t)+\frac{D_r(r,t)}{\sqrt{2D(r,t)}}p(r,r,t)&=&\frac{\mu-\lambda(r,t)}{\sqrt{2D(r,t)}}
\end{eqnarray}
alternatively, this equation can be written as
\begin{eqnarray}
\frac{\mathrm{d}}{\mathrm{d}r}\left(\sqrt{2D(r,t)}p(r,r,t)\right)&=&\frac{\mu-\lambda(r,t)}{\sqrt{2D(r,t)}}
\end{eqnarray}
Integrating the equation above, and from the fact that $p(0,0,t)=0$, the kernel equation become
\begin{eqnarray}
p_t(r,s,t) &=& D(r,t)p_{rr}(r,s,t)-\left(D(s,t)p(r,s,t)\right)_{ss}\nonumber\\
&&-\left(\mu-\lambda(r,t)\right) p(r,s,t)\label{kg1}\\
p(r,r,t) &=& \frac{1}{2\sqrt{D(r,t)}}\int_0^r\! \frac{\mu-\lambda(\tau,t)}{\sqrt{D(\tau,t)}} \,\mathrm{d}\tau\label{kg2}\\
p(0,s,t) &=& 0\label{kg3}
\end{eqnarray}

The result from this section could be stated as follow.
\begin{theorem}
Let $p(r,s,t)$ be the solution of system \eqref{kg1}-\eqref{kg3}. Then for any initial condition $\tilde{c}_0(r)\in\mathbb{L}^2(0,1)$ system \eqref{err1}-\eqref{err3} with $p_1(r,t)$ and $p_{10}(t)$ are given by \eqref{gain1} and \eqref{gain2}, respectively, has a unique classical solution $\tilde{c}(r,t)\in\mathbb{C}^{2,1}\left((0,1)\times(0,\infty)\right)$. Additionally, the origin $\tilde{c}(r,t)\equiv0$ is exponentially stable in the $\mathbb{L}^2(0,1)$-norm.
\end{theorem}

\section{Proof of Theorem 1}

The existence of a kernel function satisfy \eqref{kg1}-\eqref{kg3} is hard to be proved due to the state dependency of of the diffusion coefficient, as can be seen from the second term of the right hand side of \eqref{kg1}. In order to handle this problem, we can change the equation \eqref{kg1} into a standard form. Let us define
\begin{eqnarray}
\breve{p}(\bar{r},\bar{s},t) &=& D(s,t)p(r,s,t)\\
\bar{r} &=& \phi(r) = \sqrt{D(0,t)}\int_0^r\! \frac{1}{\sqrt{D(\tau,t)}} \,\mathrm{d}\tau\\
\bar{s} &=& \phi(s) = \sqrt{D(0,t)}\int_0^s\! \frac{1}{\sqrt{D(\tau,t)}} \,\mathrm{d}\tau
\end{eqnarray}
Using these definitions, we compute\footnote{for simplicity $\breve{p}=\breve{p}(\bar{r},\bar{s},t)$ and $p=p(r,s,t)$.}
\begin{eqnarray}
\breve{p}_t &=& D(s,t)p_t+\frac{D_t(s,t)}{D(s,t)}\breve{p}\label{p1}\\
\breve{p}_{\bar{r}} &=& D(s,t)p_r\frac{\sqrt{D(r,t)}}{\sqrt{D(0,t)}}\label{p2}\\
\breve{p}_{\bar{r}\bar{r}} &=& D(s,t)p_{rr}\frac{D(r,t)}{D(0,t)}+\frac{D_r(r,t)}{2\sqrt{D(r,t)D(0,t)}}\breve{p}_{\bar{r}}\label{p3}\\
\breve{p}_{\bar{s}} &=& D_s(s,t)p\frac{\sqrt{D(s,t)}}{\sqrt{D(0,t)}}+D(s,t)p_s\frac{\sqrt{D(s,t)}}{\sqrt{D(0,t)}}\label{p4}\\
\breve{p}_{\bar{s}\bar{s}} &=& D_{ss}(s,t)p\frac{D(s,t)}{D(0,t)}+2D_s(s,t)p_s\frac{D(s,t)}{D(0,t)}\nonumber\\
&&+D(s,t)p_{ss}\frac{D(s,t)}{D(0,t)}+\frac{D_s(s,t)}{2\sqrt{D(s,t)D(0,t)}}\breve{p}_{\bar{s}}\label{p5}
\end{eqnarray}
Rearrange \eqref{p1}, \eqref{p3}, and \eqref{p5}, we have
\begin{eqnarray}
\left(\breve{p}_t-\frac{D_t(s,t)}{D(s,t)}\breve{p}\right)\frac{1}{D(s,t)} &=& p_t\label{pio1}\\
\left(\breve{p}_{\bar{r}\bar{r}}-\frac{D_r(r,t)}{2\sqrt{D(r,t)D(0,t)}}\breve{p}_{\bar{r}}\right)\frac{D(0,t)}{D(s,t)} &=& D(r,t)p_{rr}\label{pio2}\\
\left(\breve{p}_{\bar{s}\bar{s}}-\frac{D_s(s,t)}{2\sqrt{D(s,t)D(0,t)}}\breve{p}_{\bar{s}}\right)\frac{D(0,t)}{D(s,t)} &=& \left(D(s,t)p\right)_{ss}\label{pio3}
\end{eqnarray}
Plugging \eqref{pio1}-\eqref{pio3} into \eqref{kg1}, we have
\begin{eqnarray}
\breve{p}_t &=& D(0,t)\left(\breve{p}_{\bar{r}\bar{r}}-\breve{p}_{\bar{s}\bar{s}}\right)+\left(\frac{D_t(s,t)}{D(s,t)}-(\mu-\lambda(r,t))\right)\breve{p}\nonumber\\
&&-\left(\frac{D_r(r,t)}{2}\sqrt{\frac{D(0,t)}{D(r,t)}}\breve{p}_{\bar{r}}-\frac{D_s(s,t)}{2}\sqrt{\frac{D(0,t)}{D(s,t)}}\breve{p}_{\bar{s}}\right)\label{wow}
\end{eqnarray}
The boundary conditions become
\begin{eqnarray}
\breve{p}(\bar{r},\bar{r},t) &=& \frac{1}{2}\sqrt{\frac{D(r,t)}{D(0,t)}}\int_0^{\bar{r}}\! \left(\mu-\lambda(\phi^{-1}(\tau),t)\right) \,\mathrm{d}\tau\\
\breve{p}(0,\bar{s},t) &=& 0
\end{eqnarray}
The advection term in \eqref{wow} can be eliminated like in the second section of this paper using the following transformation
\begin{eqnarray}
\bar{p}(\bar{r},\bar{s},t) &=& \left(D(r,t)D(s,t)\right)^{-\frac{1}{4}}\breve{p}(\bar{r},\bar{s},t)\label{panjang}
\end{eqnarray}
Computing the derivatives of \eqref{panjang} with respect to $t$, $\bar{r}$, and $\bar{s}$, we have\footnote{for simplicity $\bar{p}=\bar{p}(\bar{r},\bar{s},t)$.}
\begin{eqnarray}
\bar{p}_t &=& \left(D(r,t)D(s,t)\right)^{-\frac{1}{4}}\breve{p}_t-\frac{1}{4D(r,t)D(s,t)}\frac{\partial}{\partial t}\left(D(r,t)D(s,t)\right)\bar{p}\label{satu}\\
\bar{p}_{\bar{r}} &=& \left(D(r,t)D(s,t)\right)^{-\frac{1}{4}}\breve{p}_{\bar{r}}-\frac{D_r(r,t)}{4\sqrt{D(r,t)D(0,t)}}\bar{p}\\
\bar{p}_{\bar{r}\bar{r}} &=& \left(D(r,t)D(s,t)\right)^{-\frac{1}{4}}\breve{p}_{\bar{r}\bar{r}}-\frac{D_r(r,t)}{2\sqrt{D(r,t)D(0,t)}}\bar{p}_{\bar{r}}\nonumber\\
&&-\left(\frac{1}{4}\frac{\partial}{\partial r}\left(\frac{D_r(r,t)}{\sqrt{D(r,t)}}\right)\frac{1}{\sqrt{D(0,t)}}+\frac{D_r(r,t)^2}{16D(r,t)D(0,t)}\right)\bar{p}\label{dua}\\
\bar{p}_{\bar{s}} &=& \left(D(r,t)D(s,t)\right)^{-\frac{1}{4}}\breve{p}_{\bar{s}}-\frac{D_s(s,t)}{4\sqrt{D(s,t)D(0,t)}}\bar{p}\\
\bar{p}_{\bar{s}\bar{s}} &=& \left(D(r,t)D(s,t)\right)^{-\frac{1}{4}}\breve{p}_{\bar{s}\bar{s}}-\frac{D_s(s,t)}{2\sqrt{D(s,t)D(0,t)}}\bar{p}_{\bar{s}}\nonumber\\
&&-\left(\frac{1}{4}\frac{\partial}{\partial s}\left(\frac{D_s(s,t)}{\sqrt{D(s,t)}}\right)\frac{1}{\sqrt{D(0,t)}}+\frac{D_s(s,t)^2}{16D(s,t)D(0,t)}\right)\bar{p}\label{tiga}
\end{eqnarray}
Rearrange \eqref{satu}, \eqref{dua}, and \eqref{tiga}, we have
\begin{eqnarray}
\breve{p}_t &=& \left(D(r,t)D(s,t)\right)^{\frac{1}{4}}\left(\bar{p}_t+\frac{1}{4D(r,t)D(s,t)}\frac{\partial}{\partial t}\left(D(r,t)D(s,t)\right)\bar{p}\right)\label{lam1}\\
\breve{p}_{\bar{r}\bar{r}} &=& \left(D(r,t)D(s,t)\right)^{\frac{1}{4}}\left(\bar{p}_{\bar{r}\bar{r}}+\frac{D_r(r,t)}{2\sqrt{D(r,t)D(0,t)}}\bar{p}_{\bar{r}}\right)\nonumber\\
&&+L(r,t)\left(D(r,t)D(s,t)\right)^{\frac{1}{4}}\bar{p}\label{lam2}\\
\breve{p}_{\bar{s}\bar{s}} &=& \left(D(r,t)D(s,t)\right)^{\frac{1}{4}}\left(\bar{p}_{\bar{s}\bar{s}}+\frac{D_s(s,t)}{2\sqrt{D(s,t)D(0,t)}}\bar{p}_{\bar{s}}\right)\nonumber\\
&&+L(s,t)\left(D(r,t)D(s,t)\right)^{\frac{1}{4}}\bar{p}\label{lam3}
\end{eqnarray}
where
\begin{eqnarray}
L(y,t) &=& \frac{1}{4}\frac{\partial}{\partial y}\left(\frac{D_y(y,t)}{\sqrt{D(y,t)}}\right)\frac{1}{\sqrt{D(0,t)}}+\frac{D_y(y,t)^2}{16D(y,t)D(0,t)}
\end{eqnarray}
Substituting \eqref{lam1}-\eqref{lam3} into \eqref{wow}, we have
\begin{eqnarray}
\bar{p}_t(\bar{r},\bar{s},t) &=& D(0,t)\left(\bar{p}_{\bar{r}\bar{r}}(\bar{r},\bar{s},t)-\bar{p}_{\bar{s}\bar{s}}(\bar{r},\bar{s},t)\right)+\bar{\lambda}(r,s,t)\bar{p}(\bar{r},\bar{s},t)\label{psi1}
\end{eqnarray}
where
\begin{eqnarray}
\bar{\lambda}(r,s,t) &=& -\frac{D_r(r,t)^2}{8D(r,t)}+\frac{D_s(s,t)^2}{8D(s,t)}+D(0,t)\left(L(r,t)-L(s,t)\right)\nonumber\\
&&-\frac{1}{4D(r,t)D(s,t)}\frac{\partial}{\partial t}\left(D(r,t)D(s,t)\right)+\frac{D_t(s,t)}{D(s,t)}\nonumber\\
&&-(\mu-\lambda(r,t))
\end{eqnarray}
The boundary conditions become
\begin{eqnarray}
\bar{p}(\bar{r},\bar{r},t) &=& \frac{1}{2\sqrt{D(0,t)}}\int_0^{\bar{r}}\! \left(\mu-\lambda(\phi^{-1}(\tau),t)\right) \,\mathrm{d}\tau\label{psi2}\\
\bar{p}(0,\bar{s},t) &=& 0\label{psi3}
\end{eqnarray}

We prove the existence of a kernel function satisfy \eqref{psi1} with boundary conditions \eqref{psi2}-\eqref{psi3} using the method of successive approximation. First, we convert the differential equation into an integral equation. We introduce the following change of variables
\begin{eqnarray}
\bar{p}(\bar{r},\bar{s},t) &=& \psi(\xi,\eta,t)\\
\xi &=& \bar{r}+\bar{s}\\
\eta &=& \bar{r}-\bar{s}
\end{eqnarray}
This transformations give
\begin{eqnarray}
\bar{p}_{\bar{r}}(\bar{r},\bar{s},t) &=& \psi_{\xi}(\xi,\eta,t)+\psi_{\eta}(\xi,\eta,t)\\
\bar{p}_{\bar{r}\bar{r}}(\bar{r},\bar{s},t) &=& \psi_{\xi\xi}(\xi,\eta,t)+2\psi_{\xi\eta}(\xi,\eta,t)+\psi_{\eta\eta}(\xi,\eta,t)\\
\bar{p}_{\bar{s}}(\bar{r},\bar{s},t) &=& \psi_{\xi}(\xi,\eta,t)-\psi_{\eta}(\xi,\eta,t)\\
\bar{p}_{\bar{s}\bar{s}}(\bar{r},\bar{s},t) &=& \psi_{\xi\xi}(\xi,\eta,t)-2\psi_{\xi\eta}(\xi,\eta,t)+\psi_{\eta\eta}(\xi,\eta,t)
\end{eqnarray}
Thus, from \eqref{psi1} and \eqref{psi2}-\eqref{psi3} we have
\begin{eqnarray}
\psi_{\xi\eta}(\xi,\eta,t) &=& \frac{1}{4D(0,t)}\left(\psi_t(\xi,\eta,t)-\bar{\lambda}(\phi^{-1}(\xi,\eta),t)\psi(\xi,\eta,t)\right)\label{ade}\\
\psi(\xi,0,t) &=& \frac{1}{2\sqrt{D(0,t)}}\int_0^{\frac{\xi}{2}}\! \left(\mu-\lambda(\phi^{-1}(\tau),t)\right) \,\mathrm{d}\tau\\
\psi(\xi,-\xi,t) &=& 0
\end{eqnarray}
Integrating \eqref{ade} with respect to $\eta$ from $0$ to $\eta$, we get
\begin{eqnarray}
\psi_{\xi}(\xi,\eta,t) &=& \frac{1}{4\sqrt{D(0,t)}}\left(\mu-\lambda\left(\phi^{-1}\left(\frac{\xi}{2}\right),t\right)\right)\nonumber\\
&&+\frac{1}{4D(0,t)}\int_0^{\eta}\!\left(\psi_t(\xi,s,t)-\bar{\lambda}(\phi^{-1}(\xi,s),t)\psi(\xi,s,t)\right)\,\mathrm{d}s
\end{eqnarray}
Next, integrating the above equation with respect to $\xi$ from $-\eta$ to $\xi$, yields
\begin{eqnarray}
\psi(\xi,\eta,t) &=& \frac{1}{4\sqrt{D(0,t)}}\left(\mu-\lambda\left(\phi^{-1}\left(\frac{\xi}{2}\right),t\right)\right)(\xi+\eta)\nonumber\\
&&+\frac{1}{4D(0,t)}\int_{-\eta}^{\xi}\int_0^{\eta}\!\bar{\psi}_t(\tau,s,t)\,\mathrm{d}s\mathrm{d}\tau\label{integral}
\end{eqnarray}
where
\begin{eqnarray}
\bar{\psi}_t(\tau,s,t)= \psi_t(\tau,s,t)-\bar{\lambda}(\phi^{-1}(\tau,s),t)\psi(\tau,s,t)
\end{eqnarray}
We are ready to prove the existence of a kernel function satisfy \eqref{integral} using the method of successive approximation. Let us set up the recursive formula for \eqref{integral} as follow
\begin{eqnarray}
\psi^0(\xi,\eta,t) &=& \frac{1}{4\sqrt{D(0,t)}}\left(\mu-\lambda\left(\phi^{-1}\left(\frac{\xi}{2}\right),t\right)\right)(\xi+\eta)\nonumber\\
&\ll& C\left(1-\frac{t}{t_f}\right)^{-1}\left(\xi+\eta\right)\\
\psi^{n+1}(\xi,\eta,t) &=& \frac{1}{4D(0,t)}\int_{-\eta}^{\xi}\int_0^{\eta}\!\bar{\psi}_t^n(\tau,s,t)\,\mathrm{d}s\mathrm{d}\tau\label{ming}
\end{eqnarray}
for $t<t_f$, $C>0$, and $n\geq0$. The symbol $\ll$ denotes domination. If \eqref{ming} converges, the solution can be written as follows
\begin{eqnarray}
\psi(\xi,\eta,t) = \sum_{n=0}^{\infty} \psi^n(\xi,\eta,t)\label{series}
\end{eqnarray}
For $n=1$, we have
\begin{eqnarray}
\psi^1(\xi,\eta,t) &\ll& CC_0\left(1-\frac{t}{t_f}\right)^{-2}\frac{\xi\eta\left(\xi+\eta\right)}{2}
\end{eqnarray}
where
\begin{eqnarray}
C_0 = \max_{t_0\leq t\leq t_f}\left|\frac{1}{4D(0,t)}\right|
\end{eqnarray}
Using induction, we have
\begin{eqnarray}
\psi^n(\xi,\eta,t) &\ll& CC_0^n\left(1-\frac{t}{t_f}\right)^{-n-1}\frac{\xi^n\eta^n\left(\xi+\eta\right)}{(n+1)!}
\end{eqnarray}
Therefore, the series \eqref{series} could be proved to be absolutely and uniformly convergent.

\end{document}